\documentclass[a4paper,11pt]{article}
\usepackage[english]{babel}
\usepackage{latexsym,amsfonts,amsmath,enumerate,graphics,enumerate,amsthm,tikz,hyperref,curves}
\usepackage[affil-it]{authblk}

\let\epsilon=\varepsilon
\newtheorem{theorem}{Theorem}[section]
\newtheorem{lemma}[theorem]{Lemma}

\newtheorem{proposition}[theorem]{Proposition}

\theoremstyle{definition}

\newtheorem{remark}[theorem]{Remark}

\newcommand{\kommentaar}[1]{\begin{item}\item[]{\small\textsf{#1}}\end{itemize}}

\let\epsilon=\varepsilon
\let\phi=\varphi
\let\theta=\vartheta

\def\R{\mathbb{R}}

\newcommand{\comment}[1]{}

\newcommand{\spn}{\mathrm{span}}

\makeatletter
\def\@maketitle{%
  \newpage
  \null
  \vskip 2em%
  \begin{center}%
  \let \footnote \thanks
    {\Large\bfseries \@title \par}%
    \vskip 1.5em%
    {\normalsize
      \lineskip .5em%
      \begin{tabular}[t]{c}%
        \@author
      \end{tabular}\par}%
    \vskip 1em%
    {\normalsize \@date}%
  \end{center}%
  \par
  \vskip 1.5em}
\makeatother

\begin{document}

\title{ \LARGE An order theoretic characterization of spin factors}

\author{Bas Lemmens%
\thanks{Email: \texttt{B.Lemmens@kent.ac.uk}}}
\affil{School of Mathematics, Statistics \& Actuarial Science,
University  of Kent, Canterbury, Kent CT2 7NX, UK.}

\author{Mark Roelands%
\thanks{Email: \texttt{mark.roelands@gmail.com}}}
\affil{Unit for BMI, North-West University, Private Bag X6001-209, Potchefstroom 2520, South Africa.}

\author{Hent van Imhoff%
\thanks{Email: \texttt{h.van.imhoff@math.leidenuniv.nl}}}
\affil{Mathematical Institute, Leiden University, 
P.O. Box 9512, 2300 RA Leiden, the Netherlands.}

\date{}
\maketitle

\begin{abstract} 
The famous Koecher-Vinberg theorem characterizes the Euclidean Jordan algebras among the finite dimensional order unit spaces as the ones that have a symmetric cone.  Recently Walsh gave an alternative characterization of the Euclidean Jordan algebras. He showed that the Euclidean Jordan algebras correspond to the finite dimensional order unit spaces $(V,C,u)$ for which there exists a bijective map $g\colon C^\circ\to C^\circ$ with the property that $g$ is antihomogeneous,   
i.e., $g(\lambda x) =\lambda^{-1}g(x)$ for all $\lambda>0$ and $x\in C^\circ$, and $g$ is an order-antimorphism, i.e., $x\leq_C y$ if and only if $g(y)\leq_C g(x)$. In this paper we make a first step towards  extending this order theoretic characterization to infinite dimensional JB-algebras. We show that if $(V,C,u)$ is a complete order unit space with a strictly convex cone and $\dim V\geq 3$, then there exists a bijective antihomogeneous order-antimorphism $g\colon C^\circ\to C^\circ$ if and only if $(V,C,u)$ is a spin factor.
\end{abstract}

{\small {\bf Keywords:} Spin factors, order-antimorphisms, order unit spaces, JB-algebras, symmetric Banach-Finsler manifolds}

{\small {\bf Subject Classification:} Primary 17C65; Secondary 46B40}

\section{Introduction}\label{sec1}
Let $C$ be a cone in a real vector space $V$, so $C$ is convex, $\lambda C\subseteq C$ for all $\lambda \geq 0$ and $C\cap -C =\{0\}$. Then $C$ induces a partial ordering $\leq_C$ on $V$ by $x\leq_Cy$ if $y-x\in C$. Recall that $C$ is {\em Archimedean} if for each $x\in V$  and $y\in C$ with $nx\leq_C y$ for all $n=1,2,\ldots$, we have that $x\leq_C0$. Moreover, $u\in C$ is said to be an {\em order unit} if for each $x\in V$ there exists $\lambda\geq 0$ such that $x\leq_C \lambda u$. The triple $(V,C,u)$ is called an {\em order unit space} if $C$ is an Archimedean cone with order unit $u$. An order unit space can be equipped with the so called {\em order unit norm}, 
\[
\|x\|_u:=\inf\{\lambda>0\colon -\lambda u\leq_C x\leq_C\lambda u\}.
\] 
With respect to the order unit norm topology the cone $C$ is closed and has nonempty interior, denoted by $C^\circ$. 
In the paper we will study order unit spaces  that are complete with respect to $\|\cdot\|_u$ and which have a strictly convex cone. Recall that  a cone $C$ is {\em strictly convex} if for each linearly independent $x,y\in\partial C$, the segment $\{(1-\lambda) x+\lambda y\colon 0<\lambda<1\}$ is contained in $C^\circ$. 

An important class of complete order unit spaces are JB-algebras (with unit). A {\em Jordan algebra over $\R$} is a real vector space $A$ equipped with a commutative bilinear product $\circ$ that satisfies 
\[
a^2\circ (a\circ b) = a\circ (a^2\circ b)\mbox{\quad for all }a,b\in A.
\]
A {\em  (unital) JB-algebra} $A$ is a normed, complete Jordan algebra over $\R$ with unit $e$ satisfying
\[
\|a\circ b\|\leq\|a\|\|b\|,\quad \|a^2\|=\|a\|^2,\mbox{\quad and }\|a^2\|\leq \|a^2+b^2\|
\mbox{\quad for all $a,b\in A$}.\] 
A JB-algebra $A$ gives rise to a complete order unit space, where the cone $A_+$ is the set of squares $\{a^2\colon a\in A\}$, the unit $e$ is an order unit, and $\|\cdot\|_e$ coincides with norm of $A$, see \cite[Theorem 1.11]{AS2}. A special class of JB-algebras are spin factors. A {\em spin factor} $M$ is a real vector  space with $\dim M\geq 3$ such that $M=H\oplus \R e$ (vector space direct sum) with $(H,(\cdot\mid\cdot))$  a Hilbert space and $\R e$ the linear span of $e$, where $M$ is  given the  Jordan product
\begin{equation}\label{spin}
(a+\alpha e)\circ (b+\beta e) = \beta a +\alpha b +( (a\mid b)+\alpha\beta )e
\end{equation} 
 and norm $\|a+\lambda e\|:=\|a\|_2+|\lambda|$, with $\|\cdot\|_2$  the norm of $H$. 

The famous Koecher-Vinberg theorem (\cite{Koe} and \cite{Vin}) says that the finite dimensional JB-algebras are in one-to-one correspondence with symmetric cones, i.e., self-dual cones in a Euclidean space $V$ for which $\mathrm{Aut}(C):=\{T\in \mathrm{GL}(V)\colon T(C)=C\}$ acts transitively on $C^\circ$. As JB-algebras are merely  Banach spaces instead of Hilbert spaces, no  such characterization exists in infinite dimensions. It is, however,  interesting to ask if one could characterize the JB-algebras among the complete order unit spaces in order theoretic terms. One such characterization was obtained by Kai \cite{Kai} who characterized the symmetric cones among the homogeneous cones. More recently Walsh \cite{Wa} gave an order theoretic   characterization of finite dimensional JB-algebras using order-antimorphisms.  A map $g\colon C^\circ\to C^\circ$ is  an {\em order-antimorphism} if for each $x,y\in C^\circ$ we have that $x\leq_C y$ if and only if  $g(y)\leq_C g(x)$.  It is said to be  {\em antihomogeneous} if $g(\lambda x)= \frac{1}{\lambda}g(x)$ for all $\lambda >0$ and $x\in C^\circ$.  

Walsh \cite[Theorem 1.1]{Wa} showed that if $(V,C,u)$ is a finite dimensional order unit space, then there exists an antihomogeneous order-antimorphism $g\colon C^\circ\to C^\circ$  if and only if  $(V,C,u)$ is a JB-algebra. At present it is unknown if this characterization can be extended to infinite dimensional JB-algebras. In this paper we make the following contribution to this problem. 
\begin{theorem}\label{thm:1.1} 
If $(V,C,u)$ is a complete order unit space with a strictly convex cone and $\dim V\geq 3$, then there exists a bijective antihomogeneous order-antimorphism $g\colon C^\circ\to C^\circ$  if and only if $(V,C,u)$ is a  spin factor.
\end{theorem}

As our general approach is similar to Walsh's \cite{Wa}, we  briefly discuss the main similarities and differences. To prove that the cone is homogeneous \cite[Lemma 3.5]{Wa} Walsh uses the fact that a bijective antihomogeneous order-antimorphism is a locally Lipschitz map, and hence almost everywhere Fr\'echet differentiable by Rademacher's Theorem.  There is, however, no infinite dimensional version of Rademacher's Theorem. To overcome this difficulty, we show that a bijective antihomogeneous order-antimorphism is Gateaux differentiable at each point in  a strictly convex cone, and work with the Gateaux derivative, see Proposition \ref{prop:der}. Like Walsh we will also use ideas from metric geometry such as Hilbert's and Thompson's metrics. In particular, Walsh applies his characterization of the Hilbert's metric horofunctions \cite{Wa0}, which, at present, is not known for infinite dimensional spaces. Instead we shall show that if there exists a bijective antihomogeneous order-antimorphism on a strictly convex cone, then the cone is smooth, see Theorem \ref{thm:2.3}. This will allow us to avoid the  use of horofunctions completely, but implicitly some of Walsh's horofunction method is still present in the proof of Proposition \ref{symB}. 
 
 \section{Order-antimorphisms and symmetries} 
 
For $x,y\in V$ linearly independent we write $V(x,y):=\spn(x,y)$, $C(x,y):= V(x,y)\cap C$, and $C^\circ(x,y):=V(x,y)\cap C^\circ$.  Note that as $C$ is Archimedean, $C(x,y)$ is a closed 2-dimensional cone in $V(x,y)$, if $x\in C^\circ$.  
 
 Useful tools in the analysis are Hilbert's and Thompson's metrics on $C^\circ$. They are defined in terms of the following function. 
 For $x\in C$ and $y\in C^\circ$ let 
 \[
 M(x/y):=\inf\{\beta>0\colon x\leq_C \beta y\}.
 \]
 Note that $0\leq M(x/y)<\infty$ for all $x\in C$ and $y\in C^\circ$, if $(V,C,u)$ is an order unit space. Moreover, $M(\sigma x/\mu y) =\frac{\sigma}{\mu}M(x/y)$ for all $\sigma,\mu >0$ and $x\in C$ and $y\in C^\circ$.

 Now {\em Hilbert's metric} on $C^\circ$ is defined by 
 \[
 d_H(x,y):=\log M(x/y)+\log M(y/x),
 \]
 and {\em Thompson's metric} on $C^\circ$ is given by 
 \[
 d_T(x,y):=\max\left\{\log M(x/y),\log M(y/x)\right\}
 \]
 for $x,y\in C^\circ$. Note that $d_H(\sigma x,\mu y) = d_H(x,y)$ for all $x,y\in C^\circ$ and $\sigma,\mu>0$. So, $d_H$ is not a metric on $C^\circ$. However, for cones in an order unit space it is known \cite[Chapter 2]{LNBook} that $d_H$ is a metric between pairs of rays in $C^\circ$, as $d_H(x,y)=0$ if and only if $x=\lambda y$ for some $\lambda>0$ in that case. Thompson's metric is a metric on $C^\circ$ in an order unit space. Moreover, its topology coincides with the order unit norm topology on $C^\circ$, see \cite[Chapter 2]{LNBook}.  
 
 The following basic lemma is well known, see e.g., \cite{NS2}, and implies that each antihomogeneous order-antimorphism is an isometry under $d_H$ and $d_T$. For the reader's convenience we include the simple proof. 
 \begin{lemma}\label{lem:2.1} 
 Let $(V,C,u)$ be an order unit space. Then $g\colon C^\circ\to C^\circ$ is an antihomogeneous order-antimorphism if and only if $M(x/y)=M(g(y)/g(x))$ for all 
 $x,y \in C^\circ$. In particular, a bijective antihomogeneous order-antimorphism $g\colon C^\circ\to C^\circ$ is an isometry under $d_H$ and $d_T$, and the inverse $g^{-1}\colon C^\circ\to C^\circ$ is an antihomogeneous order-antimorphism. 
 \end{lemma}  
 \begin{proof}
Clearly, if $g\colon C^\circ\to C^\circ$ is  antihomogeneous order-antimorphism and $x\leq_C\beta y$, then $g(\beta y)\leq_C g(x)$, so that $g(y)\leq_C \beta g(x)$. This implies that $M(g(y)/g(x))\leq M(x/y)$. On the other hand, $g(y)\leq_C \beta g(x)$ implies  $g(\beta y)\leq_C g(x)$, so that $x\leq_C \beta y$ from which we conclude that $M(x/y)\leq M(g(y)/g(x))$. This shows that $M(x/y)=M(g(y)/g(x))$ for all $x,y\in C^\circ$. 

Now suppose that $M(x/y)=M(g(y)/g(x))$ for all $x,y\in C^\circ$. If $x\leq_C y$, then $M(g(y)/g(x))=M(x/y)\leq 1$, so that $g(y)\leq_C g(x)$. Likewise $g(y)\leq_C g(x)$ implies $M(x/y)=M(g(y)/g(x))\leq 1$, so that $x\leq_C y$, which shows that $g$ is an order-antimorphism. To see that $g$ is antihomogeneous note that if $x\in C^\circ$ and $\lambda>0$, then $y:=\lambda x$ satisfies $M(g(y)/g(x))=M(x/y) = 1/\lambda$ and $M(g(x)/g(y)) =M(y/x) =\lambda$. This implies that $\lambda g(y)\leq_C g(x)\leq_C \lambda g(y)$ from which we conclude that $g(\lambda x) = g(y) = \frac{1}{\lambda}g(x)$. 
 \end{proof}
 
 Every JB-algebra $A$ has a bijective  antihomogeneous order-antimorphism namely, the map $\iota\colon A^\circ_+\to A^\circ_+$ given by $\iota(a) =a^{-1}$. As shown in \cite[Section 2.4]{LRW2}, we have that $M(\iota(a)/\iota(b))=M(b/a)$ for all $a,b\in A^\circ_+$, and hence $\iota$ is a bijective antihomogeneous order-antimorphism by Lemma \ref{lem:2.1}.

 A linear functional $\phi\colon V\to \R$ is said to be {\em positive} if $\phi(C)\subseteq [0,\infty)$, and it is called {\em strictly positive} if 
 $\phi(C\setminus\{0\})\subseteq (0,\infty)$. A positive functional $\phi$ is called a {\em state} of 
 $(V,C,u)$ if $\phi(u) =1$. The set $S(V):=\{\phi\in V^*\colon \phi \mbox{ is a state}\}$ is called the {\em state space}, which is a w*-closed  convex subset of the unit ball in $V^*$,  and hence $S(V)$  is w*-compact by the Banach-Alaoglu Theorem. Moreover, as $x\leq_C\beta y$ is equivalent to $\phi(x)\leq \beta \phi(y)$ for all $\phi\in S(V)$, we get that 
 \begin{equation}\label{eq:var} 
 M(x/y)=\max_{\phi\in S(V)} \frac{\phi(x)}{\phi(y)}\mbox{\quad for all }x,y\in C^\circ.
 \end{equation}

If $(V,C,u)$ is an order unit space with a strictly convex cone, then there exists a strictly positive state on $V$ as the following lemma shows. 
\begin{lemma}\label{states} 
If $(V,C,u)$ is an order unit space with a strictly convex cone, then there exists a strictly positive state $\rho\in S(V)$. 
\end{lemma} 
\begin{proof}
Let $r\in\partial C\setminus\{0\}$. Then $C(r,u)$ is a 2-dimensional closed cone in $V$. By \cite[ A.5.1]{LNBook} there exists  an $s\in\partial C\setminus\{0\}$ such that $C(r,u)=\{\alpha r+\beta s\colon \alpha,\beta\geq 0\}$. Let $\phi$ and $\psi$ be linear functionals on $V(r,u)$ such that $\phi(r)=0=\psi(s)$, $\phi(s),\psi(r)>0$, and $\phi(u)=1=\psi(u)$. By the 
Hahn-Banach theorem we can extend $\phi$ and $\psi$ to linear functional on $V$ such that $\|\phi\|=\phi(u)=1$ and $\|\psi\|=\psi(u)=1$. It  follows from \cite[1.16 Lemma]{AS1} that $\phi,\psi\in S(V)$. 

Now let $\rho:=\frac{1}{2}(\phi+\psi)\in S(V)$. Note that $\phi(x) =0$ for $x\in C$ if and only if $x=\lambda r$ for some $\lambda\geq 0$, as $C$ is strictly convex.
Likewise, $\psi(x) =0$ for $x\in C$ if and only if $x=\lambda s$ for some $\lambda\geq 0$. This implies that $\rho(x)>0$ for all $x\in C\setminus\{0\}$. 
\end{proof}

Next we shall show that antihomogeneous order-antimorphisms on strictly convex cones map 2-dimensional subcones to 2-dimensional subcones. To prove this we use unique geodesics. Recall that given a metric space $(X,d_X)$ a {\em geodesic path} $\gamma\colon I\to X$, where $I\subseteq \R$ is a possibly unbounded  interval, is a map such that 
\[
d_X(\gamma(s),\gamma(t)) =|s-t|\mbox{\quad for all }s,t\in I.
\] 
The image $\gamma(I)$ is simply called a {\em geodesic}, and $\gamma(\R)$ is said to be a {\em geodesic line} in $(X,d_X)$. A geodesic line $\gamma$ is called {\em unique}  if for each $x$ and $y$ on $\gamma$ we have that $\gamma$ is the only geodesic line  through $x$ and $y$ in $(X,d_X)$. 

If $(V,C,u)$ is an order unit space with a strictly positive functional $\rho\in S(V)$, then $d_H$ is a metric on 
\[
\Sigma_\rho:=\{x\in C^\circ\colon \rho(x) =1\}.
\] 
Straight line segments are geodesic in the Hilbert's metic space $(\Sigma_\rho,d_H)$. Moreover, if the cone is strictly convex, then  it is well known, see for example \cite[Section 18]{Bus}, that each geodesic in the Hilbert's metic space $(\Sigma_\rho,d_H)$ is a straight line segment. 

\begin{lemma} \label{planes}
Let $(V,C,u)$ be an order unit space with a strictly convex cone, and $g\colon C^\circ\to C^\circ$ be a bijective  antihomogeneous order-antimorphism. If $x,y\in C^\circ$ are linearly independent, then $g(x)$ and $g(y)$ are linearly independent and $g$ maps $C^\circ(x,y)$ onto $C^\circ(g(x),g(y))$. 
\end{lemma}
\begin{proof}
Let $\rho\in S(V)$ be a strictly positive state, which we know exists by Lemma \ref{states}. Now define $f\colon \Sigma_\rho\to\Sigma_\rho$ by 
\[
f(x) :=\frac{g(x)}{\rho(g(x))}\mbox{\quad for all }x\in\Sigma_\rho.
\]
Then $f$ is an isometry on  $(\Sigma_\rho,d_H)$ by Lemma \ref{lem:2.1}.   If $x,y\in C^\circ$ are linearly independent, then the straight line $\ell$ through $x/\rho(x)$ and $y/\rho(y)$ intersected with $\Sigma_\rho$ is a geodesic line in $(\Sigma_\rho, d_H)$. Thus, $f(\ell\cap \Sigma_\rho)$ is also a geodesic line, and hence a straight line segment, as $C$ is strictly convex. 
In fact, its image is the intersection of the straight line  through $g(x)/\rho(g(x))$ and $g(y)/\rho(g(y))$ and $\Sigma_\rho$. It follows that  $g(x)/\rho(g(x))$ and $g(y)/\rho(g(y))$ are linearly independent and that $g$ maps $C^\circ(x,y)$ onto $C^\circ(g(x),g(y))$, as $g$ is antihomogeneous. 
\end{proof}
We note that the proof of Lemma \ref{planes} goes through if one only assumes that  $(\Sigma_\rho,d_H)$ is uniquely geodesic. 

Using this lemma we can now prove the following proposition. 
\begin{proposition}\label{prop:der}
Let $(V,C,u)$ be an order unit space with a strictly convex cone. If $g\colon C^\circ\to C^\circ$ is a bijective antihomogeneous order-antimorphism, then the following assertions hold. 
\begin{enumerate}[(1)] 
\item For each linearly independent $x,y\in C^\circ$  the restriction $g_{xy}$ of $g$ to $C^\circ(x,y)$ is a Fr\'echet differentiable map, and its Fr\'echet derivative $Dg_{xy}(z)$ at $z\in C^\circ(x,y)$ is an invertible linear map from $V(x,y)$ onto $V(g(x),g(y))$. 
\item For each $x\in C^\circ$ and $z\in V$ we have that 
\[
\Delta_x^z g(x):=\lim_{t\to 0} \frac{g(x+t z)-g(x)}{t}
\]
exists, and $-\Delta_x^zg(x)\in C$ for all $z\in C$. 
\item For each $x\in C^\circ$ we have $\Delta_x^{\lambda x}g(x) = -\lambda g(x)$ for all $\lambda\in \R$.
\end{enumerate}
\end{proposition}
\begin{proof}
Let  $x,y\in C^\circ$ be linearly independent and $g\colon C^\circ\to C^\circ$ be an antihomogeneous order-antimorphism. By Lemma \ref{planes} the restriction $g_{xy}$ of $g$ maps  $C^\circ(x,y)$ onto $C^\circ(g(x),g(y))$. The 2-dimensional closed cones $C(x,y)$ and $C(g(x),g(y))$ are order-isomorphic to $\R^2_+:=\{(x_1,x_2)\in\R^2\colon x_1,x_2\geq 0\}$, i.e., there exist linear maps $A\colon V(x,y)\to \R^2$ and $B\colon V(g(x),g(y))\to\R^2$ such that $A(C(x,y))=\R^2_+$ and $B(C(g(x),g(y)))=\R^2_+$. Thus, the map $h\colon (\R_+^2)^\circ\to(\R_+^2)^\circ$ given by $h(z) = B(g_{xy}(A^{-1}(z)))$ is a bijective antihomogeneous order-antimorphism on  $(\R^2_+)^\circ$, and hence $h$ is a $d_T$-isometry on $(\R_+^2)^\circ$. We know from \cite[Theorem 3.2]{LRW2} that $h$ is of the form: 
\[
h((z_1,z_2)) = (a_1/z_{\sigma(1)},a_2/z_{\sigma(2)})\mbox{\quad for }(z_1,z_2)\in (\R^2_+)^\circ,
\]
where $\sigma$ is a permutation on $\{1,2\}$ and $a_1,a_2>0$ are fixed. Clearly the map $h$ is Fr\'echet differentiable on $(\R^2_+)^\circ$, and hence $g_{xy}$ is Fr\'echet differentiable on $C^\circ(x,y)$. Moreover, the Fr\'echet derivative $Dh(z)$ is an invertible linear map on $\R^2$ at each $z\in (\R^2_+)^\circ$, so that  
$Dg_{xy}(z)$  an invertible linear map from $V(x,y)$ onto $V(g(x),g(y))$ for all $z\in C^\circ(x,y)$. 

To prove the second statement note that if $z$ is linearly independent of $x$, then there exists a $y\in C^\circ$ such that $z\in V(x,y)$. From (1) we get that $\Delta_x^z g(x) = Dg_{xy}(x)(z)$, as $g_{xy}$ is Fr\'echet differentiable on $C^\circ(x,y)$. Also, if $z=\lambda x$ for some $\lambda\neq 0$, then 
\[
\Delta_x^{\lambda x} g(x) = \lim_{t\to 0}\frac{g(x+t\lambda x)-g(x)}{t} = \lim_{t\to 0}\frac{-\lambda t}{t(1+\lambda t)} g(x) =-\lambda g(x), 
\] 
and $\Delta_x^0 g(x) =0$. Furthermore, if $z\in C$, then 
\[
\Delta_x^{z} g(x) = \lim_{t\to 0}\frac{g(x+tz)-g(x)}{t} \in -C,
\]
as $g$ is an order-antimorphism. This completes the proofs of (2) and (3). 
\end{proof}

Given a bijective antihomogeneous order-antimorphism $g\colon C^\circ\to C^\circ$ on a strictly convex cone $C$ in an order unit space, and $x\in C^\circ$ we define $G_x=G_{g,x}\colon V\to V$ by 
\[
G_{x}(z) :=-\Delta_x^z g(x)\mbox{\quad for all }z\in V.
\]
\begin{lemma}\label{lem:2.5}
If $x\in C^\circ$ and $G_x(x)=x$, then $g(x)=x$.
\end{lemma}
\begin{proof}
Simply note that $x=G_x(x) = -\Delta_x^x g(x)=g(x)$ by Proposition \ref{prop:der}(3).
\end{proof}
The map $G_x$ has the following property. 
\begin{proposition}\label{prop:2.6}
The map $G_x\colon V\to V$ is a bijective homogeneous order-isomorphism with inverse $G_{g^{-1},g(x)}\colon V\to V$. 
\end{proposition}
\begin{proof}
Let $z\in V(x,y)$, $x,y\in C^\circ$ linearly independent, and $\lambda\neq 0$. Then 
\[
G_x(\lambda z) = -\lim_{t\to 0} \frac{g(x+t\lambda z)-g(x)}{t}=-\lambda \lim_{t\to 0}\frac{g(x+t\lambda z)-g(x)}{\lambda t} = \lambda G_x(z).
\]

Also if $w\leq_C z$, then 
\[
G_x(w) = -\lim_{t\to 0}  \frac{g(x+tw)-g(x)}{t} \leq_C   -\lim_{t\to 0}  \frac{g(x+tz)-g(x)}{t} = G_x(z),
\]
as $x+tw\leq_C x+tz$ for all $t>0$ and $g$ is an order-antimorphism. 

To show that $G_x$ is a surjective map on $V$ let $h:= g_{xy}\circ g^{-1}_{g(x)g(y)}$. So, 
$h\colon C^\circ(g(x),g(y))\to C^\circ(g(x),g(y))$ and $h(z) =z$ for all $z\in C^\circ(g(x),g(y))$. For each $w\in V(g(x),g(y))$ we have by the chain rule that 
\[ 
w = Dh(g_{xy}(x))(w) = Dg_{xy}(x)Dg^{-1}_{g(x)g(y)}(g_{xy}(x))w  = G_x(G_{g^{-1},g(x)}(w)).
\] 

Interchanging the roles of $g$ and $g^{-1}$ we also have that $G_{g^{-1},g(x)}(G_x(v))=v$ for all $v\in V(x,y)$, and hence $G_{g^{-1},g(x)}$ is the inverse of $G_x$ on $V$.  
\end{proof}

Combining Proposition \ref{prop:2.6} and  \cite[Theorem B]{NS2}   we conclude that $G_x\in \mathrm{Aut}(C):=\{T\in \mathrm{GL}(V)\colon T(C)=C\}$ and $G_x$ is continuous with respect to $\|\cdot\|_u$ on $V$, as $\|G_x\|_u=\|G_x(u)\|_u$.  

Now for $x\in C^\circ$  define the {\em symmetry at $x$} by  
\begin{equation}\label{Sx} 
S_x:= G^{-1}_x\circ g.
\end{equation}
So, $S_x\colon C^\circ\to C^\circ$ is a bijective antihomogeneous order-antimorphism, with inverse $S_x^{-1} = g^{-1}\circ G_x$. We derive some further properties of the symmetries. Let us begin by making the following useful observation. 
\begin{lemma}\label{DSx}
Let $x\in C^\circ$ and $y\in V$ be linearly independent of $x$. Then for each $w\in V(x,y)$ we have  that 
$D(S_x)_{xy}(x)(w)=-w$. 
\end{lemma}
\begin{proof}
Note that 
\begin{eqnarray*}
D(S_x)_{xy}(x)(w)& = & \lim_{t\to 0} \frac{S_x(x+tw) - S_x(x)}{t} \\
    & = & \lim_{t\to 0} \frac{G^{-1}_x(g(x+tw)) - G^{-1}_x(g(x))}{t} \\
	& = & G^{-1}_x\left(\lim_{t\to 0} \frac{g(x+tw) - g(x)}{t}\right) \\
	& = & G^{-1}_x(-G_x(w))\\
	& = & -w,
\end{eqnarray*}
as $G^{-1}_x= G_{g^{-1},g(x)}$ is a bounded linear map on $(V,\|\cdot\|_u)$ by Proposition \ref{prop:2.6}.
\end{proof}

\begin{theorem}\label{thm:symmetries}
For each $x\in C^\circ$ we have that 
\begin{enumerate}[(1)]
\item $S_x(x)= x$. 
\item $S_x\circ S_x=\mathrm{Id}$ on $C^\circ$. 
\end{enumerate}
\end{theorem}
\begin{proof}
To prove (1) note that  for $x\in C^\circ$ we have by Propositions \ref{prop:der}(3) and \ref{prop:2.6} that 
\[
S_x(x) = G^{-1}_x(g(x)) =G_{g^{-1},g(x)}(g(x)) = g^{-1}(g(x))=x.
\]

To show  (2) let $x,y\in C^\circ$ be linearly independent. For simplicity we write $T:=(S_x)_{S_x(x)S_x(y)}$ and $S:=(S_x)_{xy}$, so $(S_x^2)_{xy} = T\circ S$ and $S,T$ are Fr\'echet differentiable on $C^\circ(x,y)$ and $C^\circ(S_x(x),S_x(y))$ respectively. 
Then using the chain rule and Lemma \ref{DSx} we find that 
\[
\Delta_x^y S_x^2(x) =  \lim_{t\to 0} \frac{T(S(x+ty)) -T(S(x))}{t} =   DT(S(x))(DS(x))(y) =  -DS(x)(y) =y.
\]

Note that $S^2_x$ is a homogeneous order-isomorphism on $C^\circ$, and hence by \cite[Theorem B]{NS2} we know that it is linear. So, it follows from the previous equality that $S^2_x=\mathrm{Id}$ on $C^\circ$. 
\end{proof}

To proceed it is useful to recall a few facts about unique geodesics for Thompson's  metric from \cite[Section 2]{LR}. If $x\in( C^\circ,d_T)$, then there are two special types of geodesic lines through $x$. There are the so-called {\em type I geodesic lines} $\gamma$ which are the images of the geodesic paths,
\begin{equation}\label{typeI} 
\gamma(t):= e^tr+e^{-t}s\mbox{\quad for $t\in \R$,}
\end{equation} 
with $r,s\in\partial C$ and $r+s =x$. The {\em type II geodesic line} $\mu$ through $x$ is  the image of the geodesic path $\mu(t);=e^t x$ with $t\in\R$. 
The type I geodesics $\gamma$ have the property   that $M(u/v)=M(v/u)$ for all $u$ and $v$ on $\gamma$, and the type II geodesics have  the property that $M(u/v)=M(v/u)^{-1}$ for all $u$ and $v$ on $\mu$. 

Each unique geodesic line in $(C^\circ, d_T)$ is either of type I or type II, see \cite[Section 2]{LR}. Moreover, the type II geodesic is always unique \cite[Proposition 4.1]{LR}, 
but the type I geodesics may not be unique.  However,  if $C$ is strictly convex, then all type I geodesic lines are unique, see \cite[Theorem 4.3]{LR}. 
\begin{lemma}\label{geo}
Let $(V,C,u)$ be an order unit space with a strictly convex cone. 
If $\gamma\colon\R\to (C^\circ,d_T)$ is a  geodesic path with $\gamma(0)=x$, and $\gamma(\R)$ is a type I geodesic line, then $S_x(\gamma(t)) =\gamma(-t)$ for all $t\in\R$.
\end{lemma}
\begin{proof}
If $\gamma\colon\R\to (C^\circ,d_T)$ is a  geodesic path with $\gamma(0)=x$, and $\gamma(\R)$ is a type I geodesic line, then there exist $r,s\in\partial C$ with $r+s=x$ and 
$\gamma(t)=e^tr+e^{-t}s$ for all $t\in\R$ by \cite[Lemma 3.7]{LR}. 
As $C$ is strictly convex, we know from \cite[Theorem 4.3]{LR} that $\gamma\colon\R\to (C^\circ,d_T)$ is a unique geodesic path. This implies  that $\hat\gamma(t):= S_x(\gamma(t))$, $t\in\R$, is also a unique geodesic path in $(C^\circ, d_T)$, as $S_x$ is an isometry under $d_T$. Moreover, as $M(S_x(y)/S_x(z))=M(z/y)$ for all $y,z\in C^\circ$, we know that 
\[
M(S_x(\gamma(t_1))/S_x(\gamma(t_2)))=M(\gamma(t_2)/\gamma(t_1))=M(\gamma(t_1)/\gamma(t_2)) = M(S_x(\gamma(t_2))/S_x(\gamma(t_1))),
\] 
so that $\hat\gamma(\R)$ is a  type I geodesic line though $x$. 

It now follows again from \cite[Lemma 3.7]{LR} that there exists $u,v\in\partial C$ such that  $u+v=x$ and $\hat\gamma(t) = e^t u+e^{-t}v$ for all $t\in\R$. 
Recall from  Proposition \ref{prop:der} that the restriction $(S_x)_{rx}$ of $S_x$ to $C^\circ(r,x)$ is Fr\'echet differentiable, and hence 
\[\hat\gamma'(0) = D(S_x)_{rx}(\gamma(0))(\gamma'(0)) = D(S_x)_{rx}(x)(r-s) = -r+s
\] 
by Lemma \ref{DSx}. But also $\hat\gamma'(0) =  u-v$. Combining this with the equalities $r+s=x=u+v$, we find that $u=s$ and $v=r$. Thus, 
$S_x(\gamma(t)) =\hat\gamma(t) = e^ts+e^{-t}r = \gamma(-t)$ for all $t\in\R$.  
\end{proof}

\begin{proposition}\label{fixedpoint}
Let $(V,C,u)$ be an order unit space with a strictly convex cone. For each $x\in C^\circ$ we have that $S_x$ has $x$ as a unique fixed point. 
\end{proposition}
\begin{proof} 
Suppose by way of contradiction that $y\in C^\circ$ is a fixed point of $S_x$ and $y\neq x$. Then $y$ is linearly independent of $x$, as $S_x$ is antihomogeneous and $S_x(x)=x$. Define $\mu:=M(x/y)^{1/2}M(y/x)^{-1/2}$ and $z:=\mu y\in C^\circ$. Then $M(x/z)=M(z/x)$ and hence there exists a  type I geodesic path $\gamma\colon \R\to (C^\circ,d_T)$ through $x$ and $z$, with $\gamma(0)=x$. From Lemma \ref{geo} it follows that $S_x(\gamma(\R))=\gamma(\R)$, As $z$ is the unique point of intersection of $\gamma(\R)$ with the  invariant ray $R_y:=\{\lambda y\colon\lambda >0\}$, we conclude that $S_x(z)=z$. This, however, contradicts Lemma \ref{geo}, as $z\neq x$.  \end{proof}

\begin{remark}
The metric space $(C^\circ,d_T)$ is a natural example of a Banach-Finsler manifold, see \cite{Nu}. So, the results in this section show that if there exists a bijective antihomogeneous order-antimorphism on $C^\circ$ in a complete order unit space with strictly convex cone, then $(C^\circ,d_T)$ is a {\em globally symmetric} Banach-Finsler manifold, in the sense that for each $x\in C^\circ$ there exists an isometry $\sigma_x\colon C^\circ\to C^\circ$ such that $\sigma_x^2 =\mathrm{Id}$ and $x$ is an isolated fixed point of $\sigma_x$. Indeed, we can take $\sigma_x=S_x$. It is interesting to understand which complete order unit spaces $(C^\circ,d_T)$  are globally symmetric Banach-Finsler manifolds. It might well be true that these are precisely the JB-algebras. 
\end{remark}
\section{Smoothness of the cone}
 Throughout this section we will assume that $\dim V\geq 3$. 

 We will show that if $(V,C,u)$ is a complete order unit space with a strictly convex cone and there exists an antihomogeneous order-antimorphism $g\colon C^\circ\to C^\circ$, then $C$ is a {\em smooth} cone, that is to say, for each $\eta\in\partial C$ with $\eta\neq 0$ there exists a unique $\phi\in S(V)$ such that $\phi(\eta)=0$. Before we prove this we make the following elementary observation. 
 \begin{lemma}\label{lem:2.2} If $(V,C,u)$ is an order unit space and $\eta\in\partial C$ with $\eta\neq 0$, then for each $x\in C^\circ$ and $y:=(1-s)\eta+sx$, with $0<s\leq 1$, we have that 
 \[
 M(x/y)=\frac{\phi(x)}{\phi(y)}=\frac{1}{s}
 \]
 for each $\phi\in S(V)$ with $\phi(\eta)=0$. 
 \end{lemma}
 \begin{proof}
 By \cite[Section 2.1]{LNBook} we know that 
 \[
 M(x/y) = \frac{\|\eta - x\|_u}{\|\eta -y\|_u} =\frac{1}{s}.
 \]
 But also $1/s = \phi(x)/\phi(y)$ for all states $\phi\in S(V)$ with $\phi(\eta)=0$. 
 \end{proof}
 
 \begin{theorem}\label{thm:2.3} 
 If $(V,C,u)$ is an order unit space with a strictly convex cone and there exists a bijective antihomogeneous order-antimorphism $g\colon C^\circ\to C^\circ$, then $C$ is a smooth cone.
 \end{theorem}
 \begin{proof}
 Let $\rho\in S(V)$ be a strictly positive state, which exists by Lemma \ref{states}. 
 Suppose by way of contradiction that there exist $\eta\in\partial C$ with $\rho(\eta)=1$ and states $\phi\neq \psi$ such that $\phi(\eta)=0=\psi(\eta)$. As $\phi\neq\psi$, there exists $x\in V$ such that $\phi(x)\neq \psi(x)$. Note that if $\alpha x+\beta \eta +\gamma u=0$ for some $\alpha,\beta,\gamma\in \R$, then $\alpha\phi(x) +\gamma = \alpha\psi(x)+\gamma=0$, which yields $\alpha =0$ and $\gamma=0$. This shows that $x$, $\eta$ and $u$ are linearly independent. 
 
 Let $W:=\spn(x,\eta,u)$ and $K:=W\cap C$. As $\dim V\geq 3$ and $u\in C^\circ$, $K$ is a 3-dimensional, strictly convex,  closed cone in $W$ containing $u$ in its interior. Let $S(W)$ be the state space of the order unit space 
$(W,K,u)$. Note that the restrictions of $\phi$, $\psi$, $\rho$ to $W$, denoted $\bar{\phi}$, $\bar{\psi}$, and $\bar{\rho}$ respectively, are in $S(W)$. Moreover $\bar{\rho}(w)>0$ for all $w\in K\setminus\{0\}$, and hence 
\[
\Omega:=\{ w\in K\colon \bar{\rho}(w)=1\}
\]
is a 2-dimensional, strictly convex, compact set, with $\eta$ in its (relative) boundary. We also know that $S(W)$ is a compact, convex subset of $W^*$.

Let $F:=\{\zeta\in S(W)\colon \zeta(\eta)=0\}$, which is a closed face of $S(W)$. As $F$ contains $\bar{\phi}$ and $\bar{\psi}$ which are not equal, $F$ is a straight line segment, say $[\tau,\nu]$ with $\tau\neq \nu$. 
Let $x,y\in\partial \Omega$ be such that $u$ is between the straight line segments $[\eta,x]$ and $[\eta,y]$, as in Figure \ref{fig:1}. 
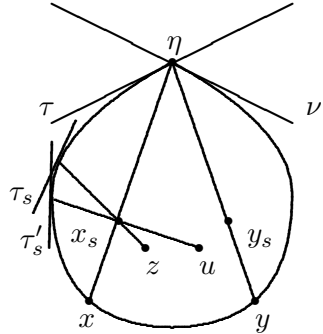
\begin{figure}[h]
\begin{center}
\begin{picture}(200, 120)
\thicklines
   \curve(100,100, 55,50, 100,0, 145,50, 100,100)
\put(98,105){$\eta$}
\put(110,20){$u$}
\put(110,30){\circle*{3.0}}
\put(90,20){$z$}
\put(90,30){\circle*{3.0}}
\put(131,0){$y$}
\put(131,10){\circle*{3.0}}
  \curve(131,10, 100,100)
\put(65,0){$x$}
\put(69,10){\circle*{3.0}}
  \curve(69,10, 100,100) 

\put(62,32){$x_s$}
\put(80,40){\circle*{3.0}}

 \curve(90,30, 58,62) 
 \curve(48,43, 64,78)

 \curve(110,30, 56,48) 
 \put(39,48){$\tau_s$}
 \curve(54,30, 55,70) 
\put(43,30){$\tau'_s$}

\put(128,32){$y_s$}
\put(121,40){\circle*{3.0}}


\put(100,100){\circle*{3.0}}
\put(55,77){\line(2,1){90}}
\put(50,80){$\tau$}
\put(145,77){\line(-2,1){90}}
\put(150,80){$\nu$}
   \end{picture}
   \caption{Point of non-smoothness}
   \label{fig:1}
  \end{center}
  \end{figure}

Now let $z\in\Omega\cap C^\circ$ also be between the segments $[\eta,x]$ and $[\eta,y]$ such that $\spn(z,\eta,u) = W$. For $0<s<1$, let $x_s:=(1-s)\eta+sx$ and $y_s:=(1-s)\eta+sy$. 
By Lemma \ref{lem:2.2} there exists $\tau_s,\tau'_s\in S(W)$  such that 
\[
M(z/x_s)=\frac{\tau_s(z)}{\tau_s(x_s)}\mbox{\quad and\quad }M(u/x_s) =\frac{\tau_s'(u)}{\tau'_s(x_s)}
\]
for $0<s<1$.  

Then 
\[
\tau_s'(z) =\frac{\tau_s'(z)}{\tau_s'(x_s)}\frac{\tau_s'(x_s)}{\tau_s'(u)}\leq M(z/x_s)M(u/x_s)^{-1}\leq\frac{\tau_s(z)}{\tau_s(x_s)}\frac{\tau_s(x_s)}{\tau_s(u)}\leq \tau_s(z) 
\]
for all $0<s<1$. As $\tau_s(z)\to\tau(z)$ and $\tau_s'(z)\to \tau(z)$ as $s\to 0$, we conclude that 
\[
\lim_{s\to 0} M(z/x_s)M(u/x_s)^{-1} =\tau(z).
\]
In the same way it can be shown that 
\[
\lim_{s\to 0} M(z/y_s)M(u/y_s)^{-1} =\nu(z).
\]

We will now show that $\tau(z)=\nu(z)$, which implies that $\tau=\nu$, as $\tau(\eta)=\nu(\eta)=0$, $\tau(u)=\nu(u)=1$ and $\spn(z,\eta,u)=W$. 
This gives the desired contradiction. To prove the equality we use the symmetry $S_u\colon C^\circ\to C^\circ$ at $u$. Let $f\colon \Sigma_\rho\to\Sigma_\rho$ be given by 
\[
f(v) =\frac{S_u(u)}{\rho(S_u(v))}\mbox{\quad for all }v\in \Sigma_\rho=\{w\in C^\circ\colon \rho(w) =1\}.
\] 
Thus, $f$ is an isometry on $(\Sigma_\rho,d_H)$. As $C$ is strictly convex, the segments $(x,\eta)$ and $(y,\eta)$ are unique geodesic lines  in $(\Sigma_\rho,d_H)$. So,  $f((x,\eta))$ and $f((y,\eta))$ are unique geodesic lines, and hence there exist $x',y',\zeta_1,\zeta_2\in\partial \Sigma_\rho$ so that $f((x,\eta))=(x',\zeta_1)$ with $\lim_{s\to 0}f(x_s) = \zeta_1$, and 
$f((y,\eta))=(y',\zeta_2)$ with $\lim_{s\to 0}f(y_s) = \zeta_2$. 

We claim that $\zeta_1=\zeta_2$. Suppose by way of contradiction that $\zeta_1\neq\zeta_2$. Then using \cite[Theorem 5.2]{KN} we know that there exists a constant $C_0<\infty$ such that 
\begin{equation}\label{grom}
\limsup_{s\to 0}\, d_H(f(x_s),u)+d_H(f(y_s),u)-d_H(f(x_s),f(y_s))\leq C_0,
\end{equation}
as $\Sigma_\rho$ is strictly convex. 

However,  we know (see \cite[Section 2.1]{LNBook}) that 
\[
d_H(x_s,y_s) =\log \frac{\|y_s -w'_s\|}{\|x_s -w'_s\|}\frac{\|x_s -v'_s\|}{\|y_s -v'_s\|}
\] 
for all $0<s<1$, where $w_s',v_s'\in\partial \Omega$. Let $w_s,v_s$ be on the lines $\ell_1$ and $\ell_2$  as in Figure \ref{fig:2}, where $\ell_1$ and $\ell_2$ are fixed. For $s>0$ sufficiently small 
\[
\frac{\|y_s -w'_s\|}{\|x_s -w'_s\|}\frac{\|x_s -v'_s\|}{\|y_s -v'_s\|}\leq \frac{\|y_s -w_s\|}{\|x_s -w_s\|}\frac{\|x_s -v_s\|}{\|y_s -v_s\|}.
\]
By projective invariance of the cross-ratio we know there exists $C_1<\infty$ such that 
\[
\frac{\|y_s -w_s\|}{\|x_s -w_s\|}\frac{\|x_s -v_s\|}{\|y_s -v_s\|}= C_1\mbox{\quad for all $s>0$ sufficiently small.}
\]
Thus,  $\lim\sup_{s\to 0} d_H(x_s,y_s)\leq \log C_1$. 
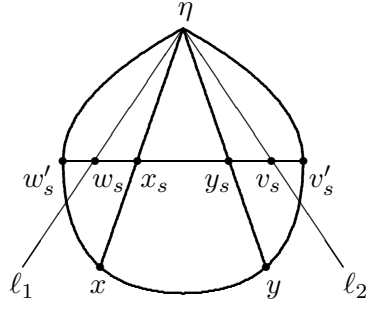
\begin{figure}[h]
\begin{center}
\begin{picture}(200, 110)
\thicklines
   \curve(100,100, 55,50, 100,0, 145,50, 100,100)
\put(98,105){$\eta$}
\put(131,0){$y$}
\put(131,10){\circle*{3.0}}
  \curve(131,10, 100,100)
\put(65,0){$x$}
\put(69,10){\circle*{3.0}}
  \curve(69,10, 100,100) 

\put(84,40){$x_s$}
\put(83, 50){\circle*{3.0}}

\put(66,40){$w_s$}
\put(67, 50){\circle*{3.0}}
\put(40,40){$w'_s$}
\put(55, 50){\circle*{3.0}}

\put(147,40){$v'_s$}
\put(145, 50){\circle*{3.0}}
\put(127,40){$v_s$}
\put(133, 50){\circle*{3.0}}

\put(108,40){$y_s$}
\put(117,50){\circle*{3.0}}
\thinlines
\put(100,100){\line(-2,-3){60}}

\put(100,100){\line(2,-3){60}}

\put(35,0){$\ell_1$}
\put(160,0){$\ell_2$}

\put(55,50){\line(1,0){90}}


   \end{picture}
   \caption{cross-ratios}
   \label{fig:2}
  \end{center}
  \end{figure}

As $f$ is an isometry under $d_H$ with $f(u)=u$, we deduce  that 
\[
d_H(f(x_s),u)+d_H(f(y_s),u)-d_H(f(x_s),f(y_s)) = d_H(x_s,u)+d_H(y_s,u)-d_H(x_s,y_s)\to\infty,
\]
as $s\to 0$. This contradicts (\ref{grom}), and hence  $\zeta_1=\zeta_2$.

Now note that 
\begin{eqnarray*}
\tau(z) & = & \lim_{s\to 0} M(z/x_s)M(u/x_s)^{-1}\\ 
   & = & \lim_{s\to 0} M(S_u(x_s)/S_u(z))M(S_u(x_s)/u)^{-1}\\
   & = & \lim_{s\to 0} M(f(x_s)/S_u(z))M(f(x_s)/u)^{-1}\\
   & = &  M(\zeta_1/S_u(z))M(\zeta_1/u)^{-1}.
\end{eqnarray*} 
Likewise $\nu(z) =   M(\zeta_2/S_u(z))M(\zeta_2/u)^{-1}$, which shows that $\tau(z) =\nu(z)$, as $\zeta_1=\zeta_2$. This completes the proof.
 \end{proof}

\begin{lemma} \label{lem:2.4}
Let $(V,C,u)$ be an order unit space with a smooth cone, $\eta\in\partial C\setminus\{0\}$, and $\phi\in S(V)$ be such that $\phi(\eta)=0$. Suppose that  $z\in C$ with $\phi(z)>0$, and for $0<s\leq 1$ let $y_s:=(1-s)\eta +s u$  and $z_s:=(1-s)z +s u$ in $C^\circ$. If $\phi_s\in S(V)$ is such that $M(z_s/y_s) =\phi_s(z_s)/\phi_s(y_s)$ for $0<s\leq 1$, then $\phi_s(\eta)\to 0$,  as $s\to 0$, and $(\phi_s)$ w*-converges  to $ \phi$. 
 \end{lemma}
 \begin{proof}
 Note that  $M(z_s/y_s)=\phi_s(z_s)/\phi_s(y_s)\geq\phi(z_s)/\phi(y_s)= \frac{1-s}{s}\phi(z) +1\to\infty$, as $s\to 0$. As $|\phi_s(z_s)|\leq \|z_s\|_u\leq (1-s)\|z\|_u +s\|u\|_u\leq \|z\|_u +1$, we deduce that $\phi_s(y_s)\to 0$ as $s\to 0$. So, 
 \[
 |\phi_s(\eta)|\leq |\phi_s(\eta) -\phi_s(y_s)|+|\phi_s(y_s)|\leq \|\eta -y_s\|_u +|\phi_s(y_s)|\to 0\mbox{\quad as } s\to0.
 \]
Now consider any subnet $(\phi_{s'})$ of $(\phi_s)$ in $S(V)$. It  has a w*-convergent subnet with limit say $ \psi$, as $S(V)$ is w*-compact. By the first part of the lemma we know that 
$\psi(\eta)=0$, and hence $\psi =\phi$, since $C$ is smooth. This shows that $(\phi_s)$ w*-converges to $\phi$.  
 \end{proof}
 
 \begin{proposition}\label{horo} 
 Let $(V,C,u)$ be an order unit space with a smooth cone, $\eta\in\partial C\setminus\{0\}$, and $\phi\in S(V)$ be such that $\phi(\eta)=0$. 
 Suppose that  $z\in C$ with $\phi(z)>0$ and for $0<s\leq 1$ let $y_s:=(1-s)\eta +s u$  and $z_s:=(1-s)z +s u$ in $C^\circ$.
 Then 
 \[
\lim_{s\to 0} M(z_s/y_s)M(u/y_s)^{-1} =\phi(z). 
 \]
 \end{proposition}
 \begin{proof}
 For $0<s\leq 1$ let $\phi_s\in S(V)$ be such that $M(z_s/y_s) =\phi_s(z_s)/\phi_s(y_s)$. So, $(\phi_s)$ w*-converges to $\phi$  by Lemma \ref{lem:2.4}. Note that 
 \begin{eqnarray*}
 M(z_s/y_s)M(u/y_s)^{-1} & \leq &  \frac{\phi_s(z_s)}{\phi_s(y_s)} \left( \frac{\phi(u)}{\phi(y_s)}\right)^{-1}\\
  & = &\frac{\phi_s(z_s)}{\hat \phi(u)} \frac{\phi(y_s)}{\phi_s(y_s)}\\
   & = & \phi_s(z_s) \frac{\phi((1-s)\eta +su)}{\phi_s((1-s)\eta+ su)}\\
    & \leq & \phi_s(z_s)
 \end{eqnarray*}
 as $\phi(\eta)=0$ and $\phi_s(\eta)\geq 0$ for all $0<s\leq 1$.  The right-hand side of the inequality converges to $\phi(z)$ as $s\to 0$, since $(\phi_s)$ w*-converges to $\phi$. 
 
 On the other hand, if we let $\psi_s\in S(V)$ be such that $M(u/y_s)=\psi_s(u)/\psi_s(y_s)$, then $(\psi_s)$ w*-converges to $\phi$ by taking $z=u$ in Lemma \ref{lem:2.4}. Moreover, 
 \begin{eqnarray*}
 M(z_s/y_s) M(u/y_s)^{-1} & \geq & \frac{\phi(z_s)}{\phi(y_s)} \left(\frac{\psi_s(u)}{\psi_s(y_s)}\right)^{-1}\\
    & = & \frac{\phi(z_s)}{\psi_s(u)}\frac{\psi_s(y_s)}{\phi(y_s)}\\
     &\geq & \phi(z_s),
  \end{eqnarray*}
  as $\psi_s(\eta)\geq 0$. The right-hand side converges to $\phi(z)$ as $s\to 0$, which completes the proof.
 \end{proof}

 \section{Proof of Theorem \ref{thm:1.1}} 
 Define 
 \[
 \mathcal{P}:=\{p\in\partial C\colon M(p/u)=\|p\|_u=1\}.
 \]
 
 \begin{lemma}\label{p'}
 If $(V,C,u)$ is an order unit space, then for each $p\in\mathcal P$ there exists  a unique $p'\in\mathcal{P}$ with $p+p'=u$. 
 \end{lemma}
 \begin{proof}
 Note that $p\leq_C M(p/u)u=u$, so that $w:=u-p\in(\partial C\setminus\{0\})\cap V(p,u)$. So, 
 \[
 M(w/u) :=\inf\{\beta>0\colon u-p\leq_C \beta u\}=\inf\{\beta>0\colon 0\leq_C (\beta-1)u+p\}=1, 
 \]
 as otherwise $p-\delta u\in C$ for some $\delta>0$. This would imply that $p= \delta u+(p-\delta u) \in C^\circ$, as $\delta u\in C^\circ$, which is impossible.  
 Thus, if we let $p':=w$, then clearly $p'$ is unique, $p'\in \mathcal{P}$ and $p+p'=u$. 
 \end{proof}
 Note that $V=\spn(\mathcal{P})$. Indeed, if $v\in V$ is linearly independent of $u$, then $V(u,v)$ is a 2-dimensional subspace with a 2-dimensional closed cone $C(u,v)$. By \cite[A.5.1]{LNBook} there exists $r,s\in\partial C$ such that $C(u,v)=\{\lambda r+\mu s\colon \lambda,\mu\geq 0\}$ and $\spn(r,s)=V(u,v)$. So, if we let $p:=M(r/u)^{-1}r$ and $q:=M(s/u)^{-1}s$, then $p,q\in\mathcal{P}$ and $v\in \spn(p,q)$.  On the other hand, if $v=\lambda u$ with $\lambda\in\R$, then $v=\lambda(p+p')$ for some $p\in\mathcal{P}$ by Lemma \ref{p'}.

 Now let  $(V,C,u)$ be an order unit space with a strictly convex cone and $\dim V\geq 3$. Suppose there exists a bijective antihomogeneous order-antimorphism $g\colon C^\circ\to C^\circ$. Then $C$ is a smooth cone by Theorem \ref{thm:2.3}. Denote by $\phi_p\in S(V)$ the unique supporting functional at $p\in\mathcal{P}$, so $\phi_p(p)=0$ and $\phi_p(p')=\phi_p(u)=1$. 
For $p\in \mathcal{P}$ define the linear form $B(p,\cdot)$ on $V$ by 
 \[
 B(p,v) :=\phi_{p'}(v)\mbox{\quad for all }v\in V.
 \]
 \begin{proposition}\label{symB} 
 If $p,q\in\mathcal{P}$, then $B(p,q)=B(q,p)$. 
 \end{proposition}
 \begin{proof}
 Let $p,q\in \mathcal{P}$ and for $0<s\leq 1$ define 
 \begin{align*}
 p_s&:=(1-s)p+su, & p'_s& := (1-s)p'+su,\\
 q_s&:=(1-s)q+su, & q'_s&:= (1-s)q'+su.
 \end{align*}
 We wish to show that $S_u(p_s) =\frac{1}{s}p_s'$ and $S_u(q_s) =\frac{1}{s}q_s'$. By interchanging the roles of $p_s$ and $q_s$ it suffices to prove the first equality. 
 
 Note that if $\beta>0$ is such that $u\leq_C \beta p_s$, then  $(1-\beta s)u\leq_C\beta (1-s)p$, so that $\beta s\geq 1$, as $p\in\partial C$ and $u\in C^\circ$. Thus, $M(u/p_s) =1/s$. The same argument shows that $M(u/p_s')=1/s$. Furthermore, it is easy to check that $M(p_s/u)=1=M(p_s'/u)$, and hence $d_T(u,p_s)=-\log s= d_T(u,p_s')$ for all $0< s\leq 1$. 
 
Let $\delta_s:= M(u/p_s)^{1/2}M(p_s/u)^{-1/2}=1/\sqrt{s}$ and put $x_s:=\delta_s p_s$ and $y_s:=\delta_s p_s'$. Then $M(x_s/u)=M(u/x_s)= 1/\sqrt{s}=M(y_s/u)=M(u/y_s)$. Thus, $x_s$ and $y_s$ are on the unique type I geodesic line $\gamma$ through $u$ in $C^\circ(p,p')$. Let $\gamma\colon\R\to (C^\circ,d_T)$ be the geodesic path with $\gamma=\gamma(\R)$ and $\gamma(0)=u$. As $S_u$ is a $d_T$-isometry and $S_u(u)=u$, we find that    $d_T(u,x_s)=d_T(u,S_u(x_s)) =-\log\sqrt{s}=d_T(u,y_s)$. Using Lemma \ref{geo} and the fact that $x_s\neq y_s$, we conclude that $S_u(x_s)=y_s$. Thus,  $S_u(\delta_s p_s) =\delta_s p_s'$, which shows that $S_u(p_s) = \frac{1}{s} p_s'$. 

Now let $p,q\in\mathcal{P}$ and suppose that $q\neq p'$. Then by Proposition \ref{horo} we have that 
\begin{eqnarray*}
B(p,q) & = & \phi_{p'}(q)\\
  & = & \lim_{s\to 0} M(q_s/p_s')M(u/p_s')^{-1}\\
   & = & \lim_{s\to o} M(q_s/S_u(p_s))M(u/S_u(p_s))^{-1}\\
   & = & \lim_{s\to 0} M(p_s/S_u(q_s))M(p_s/u)^{-1}\\
   & = & \lim_{s\to 0} M(p_s/S_u(q_s)), 
\end{eqnarray*}
where we used the identity $S_u(p_s) =\frac{1}{s} p'_s$   and the fact that $S^2_u =\mathrm{Id}$ (Theorem \ref{thm:symmetries}) in the third equality. 

Likewise, 
\[
B(q,p) = \lim_{s\to 0} M(q_s/S_u(p_s)).
\]
Now using the fact that $M(p_s/S_u(q_s)) = M(q_s/S_u(p_s))$ for all $0<s\leq 1$, we deduce that $B(p,q) =B(q,p)$ if $q\neq p'$. On the other hand, if $q=p'$, then $B(p,q) =0$ and $B(q,p)=0$. \end{proof}
 We now extend $B$ linearly to $V$ by letting
 \[
 B\left(\sum_{i=1}^n \alpha_ip_i,v\right) := \sum_{i=1}^n \alpha_iB(p_i,v)\mbox{\quad for all }v\in V.
 \]
 To see that $B$ is a well-defined bilinear form suppose that $w=\sum_i\alpha_i p_i=\sum_j\beta_jq_j$ for some $\alpha_i,\beta_j\in\R$ and $p_i,q_j\in\mathcal{P}$.
 Write $v=\sum_k\gamma_k r_k$ with $r_k\in\mathcal{P}$. Then by Proposition \ref{symB} we get that 
 \begin{eqnarray*}
  \sum_i \alpha_i B(p_i,v)
   	& = &  \sum_{i,k} \alpha_i\gamma_kB(p_i,r_k)\\
	& = & \sum_{i,k}\gamma_k\alpha_i B(r_k,p_i)\\
	& = & \sum_k \gamma_kB(r_k,w).
 \end{eqnarray*}
 Likewise  $ \sum_j\beta_j B(q_j,v)= \sum_k \gamma_kB(r_k,w)$, which shows that $B$ is a well defined symmetric bilinear form on $V\times V$.  
 
 Let $H:=\spn \{p-p'\colon p\in\mathcal{P}\}$ and $\R u:=\spn(u)$. 
 \begin{lemma} \label{H+Ru} We have that  $V=H\oplus \R u$ (vector space direct sum), and $H$ is a closed subspace of $(V,\|\cdot\|_u)$.   
 \end{lemma}
\begin{proof}
 Note that for each $v\in V$ there exists $p\in\mathcal{P}$ and $\alpha,\beta\in\R$ such that $v=\alpha p+\beta p'$. So, 
 \begin{equation}\label{pp'}
 v=\frac{1}{2}(\alpha-\beta)(p-p') +\frac{1}{2}(\alpha+\beta)u,
 \end{equation}
 by Lemma \ref{p'}. This shows that $V=H+\R u$. Now let $\psi_u\colon V\to\R$ be given by $\psi_u(v) :=B(v,u)$ for all $v\in V$. Note that if $v=p-p'$, then 
 \[
 \psi_u(v) = B(p,u)-B(p',u) = \phi_{p'}(u) -\phi_p(u) = 1-1 =0,
 \]
 and hence $H\subseteq \mathrm{ker}(\psi_u)$. Moreover, $B(u,u) =B(p,u)+B(p',u)=2$. Also for $v=\alpha s +\beta u$ with $s=p-p'\in H$ we have that $\psi_u(v) =2\beta =0$ if and only if $\beta =0$. Thus, $H= \mathrm{ker}(\psi_u)$, which shows that $V=H\oplus \R u$. 
 
 To see that $H$ is closed it suffices to show that $\psi_u$ is bounded with respect to $\|\cdot\|_u$. Let $v=\alpha p+\beta p'\in V$. Then 
 \begin{equation}\label{unorm}
 \|v\|_u=\inf\{\lambda >0\colon -\lambda u\leq_C \alpha p+\beta p'\leq_C \lambda u\} =\max\{|\alpha|,|\beta|\}. 
 \end{equation}
 It follows that  
 \[
 |\psi_u(v)|\leq |\alpha|\psi_u(p)+|\beta|\psi_u(p') = |\alpha|+|\beta|\leq 2\|v\|_u,
 \]
 and hence $\psi_u$ is bounded. 
 \end{proof}
 Define a bilinear form  $(x\mid y)$ on $H$ by 
 \[
 (x\mid y):=\frac{1}{2}B(x,y)\mbox{\quad for all }x,y\in H.
 \]
 
 \begin{proof}[Proof of Theorem \ref{thm:1.1}]
 We will first show that $(H,(\cdot\mid\cdot))$ is a Hilbert space. Note that if $x\in H$, then there exists $p\in\mathcal{P}$ and $\alpha\in \R$ such that $x=\alpha(p-p')$ by (\ref{pp'}). 
 Clearly 
 \begin{equation}\label{norms}
 \|x\|_2^2=(x\mid x)= \frac{1}{2}\left(\alpha^2B(p,p-p') -\alpha^2B(p',p-p')\right) = \frac{\alpha^2}{2}(1+1) = \alpha^2 =\|x\|_u^2, 
 \end{equation}
 by (\ref{unorm}). It follows that $(x\mid x)\geq 0$ for all $x\in H$, $(x\mid x)=0$ if and only if $x=0$, and $(H,(\cdot\mid\cdot))$ is a Hilbert space, as  $H$ is closed in 
 $(V, \|\cdot\|_u)$. 

 We already know from Lemma \ref{H+Ru} that $V=H\oplus\R u$, where $(H,(\cdot\mid\cdot))$ is a Hilbert space. 
 Note that if $x=\alpha(p-p')\in H$, then $\|x+\beta u\|_u =\max\{|\alpha+\beta|,|\alpha -\beta|\}=|\alpha|+|\beta|=\|x\|_u +|\beta|$ by (\ref{unorm}). 
So, we deduce from equality  (\ref{norms}) that 
 \[
 \|x+\beta u\|_u= \|x\|_2+|\beta|\mbox{\quad for $x\in H$ and $\beta\in\R$}.
 \]
 
 It remains to show that $\{a^2\colon a\in V\} =C$, where the Jordan product is given by (\ref{spin}). Note that if $a=x+\sigma u$ where $x=\delta(p-p')\in H$ and $\sigma,\delta\in\R$, then 
 \begin{eqnarray*}
a^2 & = & 2\sigma x+((x\mid x)+\sigma^2)u\\
  & = & 2\sigma\delta(p-p') +\left(\frac{\delta^2}{2}B(p-p',p-p') +\sigma^2\right) u\\
   & = & 2\sigma\delta(p-p') + (\delta^2+\sigma^2)(p+p')\\
   & = & (\sigma+\delta)^2p + (\sigma-\delta)^2p'\in C.
 \end{eqnarray*}
 Conversely, if $v\in C$, then $v=\lambda p+\mu p'$ for some $\lambda,\mu\geq 0$ and $p,p'\in\mathcal{P}$. Let 
 \[
 w:=\sqrt{\lambda}p+\sqrt{\mu}p' = \frac{1}{2}\left((\sqrt{\lambda}-\sqrt{\mu})(p-p')+ (\sqrt{\lambda}+\sqrt{\mu})(p+p')\right).\] 
 So, 
 \[
 w^2 =\frac{1}{4}\left(2(\sqrt{\lambda}-\sqrt{\mu})(\sqrt{\lambda}+\sqrt{\mu})(p-p') + ((\sqrt{\lambda}-\sqrt{\mu})^2 + (\sqrt{\lambda}+\sqrt{\mu})^2)(p+p')\right) = \lambda p+\mu p'=v,
 \]
 which shows that $v\in\{a^2\colon a\in V\}$. 
 \end{proof}

\end{document}